\documentclass[12pt,reqno]{amsart}

\usepackage{multirow}
\usepackage{hhline}
\usepackage{float}

\usepackage{mathtools}
\usepackage{times}
\usepackage[T1]{fontenc}
\usepackage{mathrsfs}
\usepackage{latexsym}
\usepackage[dvips]{graphics}
\usepackage[titletoc, title]{appendix}
\usepackage{amsmath,amsfonts,amsthm,amssymb,amscd}
\usepackage[dvipsnames]{xcolor}
\usepackage{hyperref}
\usepackage{amsmath}
\usepackage[utf8]{inputenc}

\usepackage{color}
\usepackage{breakurl}

\usepackage{comment}
\newcommand{\bburl}[1]{\textcolor{blue}{\url{#1}}}
\newcommand{\seqnum}[1]{\href{https://oeis.org/#1}{\rm \underline{#1}}}

\usepackage{caption}

\newtheorem{thm}{Theorem}[section]

\newtheorem{cor}[thm]{Corollary}

\newtheorem{lem}[thm]{Lemma}

\newtheorem{defi}[thm]{Definition}
\newtheorem{rek}[thm]{Remark}
\newtheorem{prob}[thm]{Problem}

\usepackage[utf8]{inputenc}

\DeclareFixedFont{\ttb}{T1}{txtt}{bx}{n}{12} % for bold
\DeclareFixedFont{\ttm}{T1}{txtt}{m}{n}{12}  % for normal

\usepackage{color}
\definecolor{deepblue}{rgb}{0,0,0.5}
\definecolor{deepred}{rgb}{0.6,0,0}
\definecolor{deepgreen}{rgb}{0,0.5,0}

\usepackage{listings}

\newcommand\pythonstyle{\lstset{
language=Python,
basicstyle=\ttm,
morekeywords={self},              % Add keywords here
keywordstyle=\ttb\color{deepblue},
emph={MyClass,__init__},          % Custom highlighting
emphstyle=\ttb\color{deepred},    % Custom highlighting style
stringstyle=\color{deepgreen},
frame=tb,                         % Any extra options here
showstringspaces=false
}}

\lstnewenvironment{python}[1][]
{
\pythonstyle
\lstset{#1}
}
{}

\newcommand\pythoninline[1]{{\pythonstyle\lstinline!#1!}}

\definecolor{ao}{rgb}{0.0, 0.5, 0.0}

\numberwithin{equation}{section}

\DeclareFontFamily{U}{mathx}{}
\DeclareFontShape{U}{mathx}{m}{n}{<-> mathx10}{}
\DeclareSymbolFont{mathx}{U}{mathx}{m}{n}
\DeclareMathAccent{\widehat}{0}{mathx}{"70}
\DeclareMathAccent{\widecheck}{0}{mathx}{"71}

\begin{document}

\title{Linear Recurrences of Generalized Schreier Sets Revisited}

\author[H. V. Chu]{H\`ung Vi\d{\^e}t Chu}
\email{\textcolor{blue}{\href{mailto:hchu@wlu.edu}{hchu@wlu.edu}}}
\address{Department of Mathematics, Washington and Lee University, Lexington, VA 24450, USA}

\author[Z. L. Vasseur]{Zachary Louis Vasseur}
\email{\textcolor{blue}{\href{mailto:zachary.l.v@tamu.edu}{zachary.l.v@tamu.edu}}}
\address{Department of Mathematics, Texas A\&M University, College Station, TX 77840, USA}

\subjclass[2020]{05A19 (primary); 11B37, 11Y55, 05A15 (secondary)}

\keywords{Schreier set, recurrence, Pascal triangle}

\maketitle
 
\begin{abstract}
For $p, q\in \mathbb{N}$, a finite nonempty set $F$ is said to be $(p,q)$-\textit{Schreier} (or \textit{maximal $(p,q)$-Schreier}, respectively) if $q\min F\ge p|F|$ (or $q\min F = p|F|$, respectively). For $n\in \mathbb{N}$, let 
$$\mathcal{S}^{p/q}_{n}\ :=\ |\{F\subset\{1, 2, \ldots, n\}\,:\, q\min F\ge p|F|\mbox{ and }n\in F\}|.$$ 
Using the  Inclusion-Exclusion Principle, Beanland et al.\ proved the recurrence
$$|\mathcal{S}^{p/q}_{n}|\ =\ \sum_{k=1}^q(-1)^{k+1}\binom{q}{k}|\mathcal{S}^{p/q}_{n-k}| + |\mathcal{S}^{p/q}_{n-(p+q)}|.$$
We show that $(|\mathcal{S}^{p/q}_n|)_{n=1}^\infty$ is a subsequence with terms taken periodically from Padovan-like sequences which satisfy simple recurrence relations.
As an application, we obtain an alternative proof of the above linear recurrence. Furthermore, a similar result holds for the sequence $(|\mathcal{M}^{p/q}_{n}|)_{n=1}^\infty$ that counts maximal  $(p,q)$-Schreier sets. We end with a discussion of the relation between $(|\mathcal{S}^{p/q}_{n}|)_{n=1}^\infty$ and $(|\mathcal{M}^{p/q}_{n}|)_{n=1}^\infty$.
\end{abstract}

\tableofcontents

\section{Introduction}
A finite nonempty subset $F$ of natural numbers is called \textit{Schreier} if $\min F\ge |F|$ and is called \textit{maximal Schreier} if $\min F = |F|$. Bird \cite{Bi} observed that for all $n\ge 1$,
\begin{equation}\label{e0}|\{F\subset\mathbb{N}\,:\, F\mbox{ is Schreier and }\max F = n\}| \ =\ F_n,\end{equation}
where $F_n$ is the $n$\textsuperscript{th} Fibonacci number defined as: $F_{-1} = 1, F_0 = 0$ and $F_m = F_{m-1} + F_{m-2}$ for $m\ge 1$. 
Similarly, by \cite[Theorem 1]{C1},
\begin{equation}\label{e1p1}|\{F\subset\mathbb{N}\,:\, F\mbox{ is maximal Schreier and }\max F = n\}|\ =\ F_{n-2}, \mbox{ for all }n\ge 1.\end{equation}

Previous work has generalized these results in various directions: Beanland et al.\ studied the more general condition $q\min F\ge p|F|$ and discovered linear recurrences of higher order; Chu et al.\ connected Schreier multisets to higher-order Fibonacci sequences \cite{CIMSZ}; more recently, Beanland et al.\ counted unions of Schreier sets and provided a family of recursively defined sequences \cite{BGHH}.  

A set $F$ of natural numbers is said to be \textit{$(p,q)$-Schreier} (or \textit{maximal $(p,q)$-Schreier}, respectively) if $q\min F\ge p|F|$ (or $q\min F = p|F|$, respectively).
Let us recall \cite[Theorem 1.1]{BCF}, which states that if
\begin{equation*}\mathcal{S}^{p/q}_n := \{F \subset \mathbb{N}\,:\, F\mbox{ is }(p,q)\mbox{-Schreier and }\max F = n\}, \mbox{ with }p, q, n\in \mathbb{N},\end{equation*}
then 
\begin{equation}\label{ee1}|\mathcal{S}^{p/q}_n|\ =\ \sum_{i=1}^q (-1)^{i+1}\binom{q}{i}|\mathcal{S}^{p/q}_{n-i}| + |\mathcal{S}^{p/q}_{n-(p+q)}|.\end{equation}
Continuing the work of Beanland et al.\ \cite{BCF}, we investigate the family of sequences $(|\mathcal{S}^{p/q}_n|)_{n=1}^\infty$. 
While $(|\mathcal{S}^{1/3}_{n}|)_{n=1}^\infty$ and $(|\mathcal{S}^{1/4}_{n}|)_{n=1}^\infty$ were only recently added by the authors of the present paper to the On-Line Encyclopedia of Integer Sequences (OEIS) \cite{Sl}, these sequences are not far from the existing ones. Our main result shows that each $(|\mathcal{S}^{p/q}_{n}|)_{n=1}^\infty$ is a subsequence of a Padovan-like sequence whose linear recurrence is much simpler than \eqref{ee1}. This relation reveals an alternative proof of \eqref{ee1}, which employs characteristic polynomials instead of the Inclusion-Exclusion Principle. 

We are also interested in maximal $(p,q)$-Schreier sets: let 
\begin{equation*}\mathcal{M}^{p/q}_n := \{F \subset \mathbb{N}\,:\, F\mbox{ is maximal }(p,q)\mbox{-Schreier and }\max F = n\}, \mbox{ with }p, q, n\in \mathbb{N}.\end{equation*}
Tables \ref{skn} and \ref{smkn} collect the initial values of $(|\mathcal{S}^{p/q}_{n}|)_{n=1}^\infty$ and $(|\mathcal{M}^{p/q}_n|)_{n=1}^\infty$ for different $(p,q)$'s.

\begin{table}[H]
\centering
\begin{tabular}{ |c| c| c| c| c| c| c| c| c| c| c| c| c| c| c| c| c| c| c| c| c| }
\hline
$n$ &$1$& $2$ & $3$ & $4$ & $5$ & $6$ & $7$ & $8$ & $9$ & $10$ & $11$ & $12$ & $13$ & $14$ & $15$ \\
\hline
$|\mathcal{S}^{1/1}_{n}|$ & $1$ & $1$ & $2$ & $3$ & $5$ & $8$ & $13$ & $21$ & $34$ & $55$ & $89$ & $144$ & $233$ & $377$ & $610$ \\
\hline
$|\mathcal{S}^{1/2}_{n}|$ & $1$ & $2$ & $3$ & $5$ & $9$ & $16$ & $28$ & $49$ & $86$ & $151$ & $265$ & $465$ & $816$ & $1432$ & $2513$ \\
\hline
$|\mathcal{S}^{1/3}_{n}|$ & $1$ & $2$ & $4$ & $7$ & $12$ & $21$ & $38$ & $70$ & $129$ & $236$ & $429$ & $778$ & $1412$ & $2567$ & $4672$ \\
\hline
$|\mathcal{S}^{1/4}_{n}|$ & $1$ & $2$ & $4$ & $8$ & $15$ & $27$ & $48$ & $86$ & $157$ & $292$ & $549$ & $1034$ & $1939$ & $3613$ & $6697$ \\
\hline
$|\mathcal{S}^{2/1}_{n}|$ & $0$ & $1$ & $1$ & $1$ & $2$ & $3$ & $4$ & $6$ & $9$ & $13$ & $19$ & $28$ & $41$ & $60$ & $88$ \\
\hline
$|\mathcal{S}^{2/3}_{n}|$ & $1$ & $1$ & $2$ & $4$ & $7$ & $12$ & $20$ & $33$ & $55$ & $93$ & $159$ & $273$ & $468$ & $799$ & $1359$ \\
\hline
\end{tabular}
\caption{The first $15$ values of $(|\mathcal{S}^{p/q}_{n}|)_{n=1}^\infty$ with $(p,q) \in \{(1,1), (1,2), (1,3), (1,4), (2,1), (2,3)\}$. The sequences $(|\mathcal{S}^{1/1}_{n}|)_{n=1}^\infty$, $(|\mathcal{S}^{1/2}_{n}|)_{n=1}^\infty$, $(|\mathcal{S}^{2/1}_n|)_{n=1}^\infty$, and $(|\mathcal{S}^{2/3}_{n}|)_{n=1}^\infty$ are \seqnum{A000045}, \seqnum{A005314}, \seqnum{A078012}, and \seqnum{A099558}, respectively; the sequences $(|\mathcal{S}^{1/3}_{n}|)_{n=1}^\infty$ and $(|\mathcal{S}^{1/4}_{n}|)_{n=1}^\infty$  were newly added as \seqnum{A385106} and \seqnum{A385107}, respectively.}
\label{skn}
\end{table}

\begin{table}[H]
\centering
\begin{tabular}{ |c| c| c| c| c| c| c| c| c| c| c| c| c| c| c| c| c| c| c| c| c| }
\hline
$n$ &$1$& $2$ & $3$ & $4$ & $5$ & $6$ & $7$ & $8$ & $9$ & $10$ & $11$ & $12$ & $13$ & $14$ & $15$ \\
\hline
$|\mathcal{M}^{1/1}_n|$ & $1$ & $0$ & $1$ & $1$ & $2$ & $3$ & $5$ & $8$ & $13$ & $21$ & $34$ & $55$ & $89$ & $144$ & $233$  \\
\hline
$|\mathcal{M}^{1/2}_n|$ & $0$ & $1$ & $1$ & $1$ & $2$ & $4$ & $7$ & $12$ & $21$ & $37$ & $65$ & $114$ & $200$ & $351$ & $616$  \\
\hline
$|\mathcal{M}^{1/3}_n|$ & $0$ & $0$ & $1$ & $2$ & $3$ & $4$ & $6$ & $11$ & $22$ & $43$ & $80$ & $144$ & $257$ & $462$ & $839$  \\
\hline
$|\mathcal{M}^{1/4}_n|$ & $0$ & $0$ & $0$ & $1$ & $3$ & $6$ & $10$ & $15$ & $22$ & $35$ & $64$ & $129$ & $265$ & $529$ & $1013$\\
\hline
$|\mathcal{M}^{2/1}_n|$ & $0$ & $1$ & $0$ & $0$ & $1$ & $1$ & $1$ & $2$ & $3$ & $4$ & $6$ & $9$ & $13$ & $19$ & $28$\\
\hline
$|\mathcal{M}^{2/3}_n|$ & $0$ & $0$ & $0$ & $1$ & $2$ & $3$ & $4$ & $5$ & $7$ & $12$ & $23$ & $44$ & $80$ & $138$ & $230$\\
\hline
\end{tabular}
\caption{The first $15$ values of $(|\mathcal{M}^{p/q}_{n}|)_{n=1}^\infty$ with $(p,q) \in \{(1,1), (1,2), (1,3), (1,4), (2,1), (2,3)\}$. These sequences are \seqnum{A212804}, \seqnum{A005251}, \seqnum{A375169}, \seqnum{A385142} (newly added), \seqnum{A078012}, and \seqnum{A137357}, respectively.}
\label{smkn}
\end{table}

A keen observation reveals that the terms of $(|\mathcal{S}^{p/q}_{n}|)_{n=1}^\infty$ and $(|\mathcal{M}^{p/q}_{n}|)_{n=1}^\infty$ are periodically taken from certain Padovan-like sequences. For $(p,q)\in \mathbb{N}^2$ with $p\ge q$, define the sequence $(a_{p,q,n})_{n=0}^\infty$ (Table \ref{an}) as follows:
\begin{align*}
    a_{p,q,0} &\ =\ \cdots\ =\ a_{p, q, p-q-1} \ =\ 0,\\
    a_{p,q, p-q}&\ =\ \cdots\ =\ a_{p,q, p+q-1}\ =\ 1,\mbox{ and }\\
    a_{p,q, n}&\ =\ a_{p,q,n-q} + a_{p,q,n-p-q},\mbox{ for }n\ge p+q.
\end{align*}
If $p < q$, let 
\begin{align*}
    a_{p,q,0} &\ =\ \cdots\ =\ a_{p, q, 2p-1} \ =\ 1,\\
    a_{p,q, 2p}&\ =\ \cdots\ =\ a_{p,q, p+q-1}\ =\ 2,\mbox{ and }\\
    a_{p,q, n}&\ =\ a_{p,q,n-q} + a_{p,q,n-p-q},\mbox{ for }n\ge p+q.
\end{align*}

By changing the initial condition of the Padovan-like sequences $(a_{p, q, n})_{n=0}^\infty$, we arrive at a similar result for maximal $(p,q)$-Schreier sets. For $(p, q)\in \mathbb{N}^2$ and $p\ge q$, define the sequence $(a^{(m)}_{p,q, n})_{n=0}^\infty$ (Table \ref{am}) recursively as follows:
\begin{align*}
    a^{(m)}_{p,q,0}&\ =\ \cdots \ =\ a^{(m)}_{p,q,p-q-1}\ =\ 0, \quad a^{(m)}_{p,q,p-q}\ =\ 1,\\
    a^{(m)}_{p,q,p-q+1}&\ =\ \cdots\ =\ a^{(m)}_{p,q,p+q-1}\ =\ 0,\mbox{ and}\\
    a^{(m)}_{p, q, n}&\ =\ a^{(m)}_{p, q, n-q} + a^{(m)}_{p,q, n-p-q}. 
\end{align*}
If $p < q$, then
\begin{align*}
    a^{(m)}_{p,q,0}&\ =\ \cdots \ =\ a^{(m)}_{p,q,2p-1}\ =\ 0, \quad a^{(m)}_{p,q,2p}\ =\ 1,\\
    a^{(m)}_{p,q,2p+1}&\ =\ \cdots\ =\ a^{(m)}_{p,q,p+q-1}\ =\ 0,\mbox{ and}\\
    a^{(m)}_{p, q, n}&\ =\ a^{(m)}_{p, q, n-q} + a^{(m)}_{p,q, n-p-q}. 
\end{align*}

As we see later, for all $(p, q)\in \mathbb{N}^2$ and $n\ge 0$, 
\begin{equation}\label{eee13}
a_{p, q, n}\ =\ \sum_{i=0}^{\left\lfloor\frac{n-p+q}{p+q}\right\rfloor}\binom{\left\lfloor \frac{n-p(i+1)}{q}\right\rfloor+1}{i} \ =:\ \Psi(p, q, n).
\end{equation}
and
\begin{equation}\label{eee14}
a^{(m)}_{p,q,n}\ =\ \sum_{\substack{i=0\\ qi\equiv n\mod p}}^{\left\lfloor\frac{n}{q}\right\rfloor}\binom{i}{\frac{n-iq}{p}-2} + \begin{cases}1, &\mbox{ if }n = p-q;\\ 0, &\mbox{ otherwise}\end{cases}\ =:\ \Phi(p, q, n).
\end{equation}

Tables \ref{an} and \ref{am} record the first few values of $(a_{p,q,n})_{n=0}^\infty$ and $(a^{(m)}_{p,q,n})_{n=0}^\infty$ for different $(p,q)$'s.

\begin{table}[H]
\centering
\begin{tabular}{ |c| c| c| c| c| c| c| c| c| c| c| c| c| c| c| c| c| c| c| c| c| c|c|}
\hline
$n$ &$0$& $1$ & $2$ & $3$ & $4$ & $5$ & $6$ & $7$ & $8$ & $9$ & $10$ & $11$ & $12$ & $13$ & $14$ & $15$ & $16$ \\
\hline
$a_{1,1,n}$ & $1$ & $1$ & $2$ & $3$ & $5$ & $8$ & $13$ & $21$ & $34$ & $55$ & $89$ & $144$ & $233$ & $377$ & $610$ & $987$ & $1597$ \\
\hline
$a_{1,2,n}$ & $1$ & $1$ & $2$ & $2$ & $3$ & $4$ & $5$ & $7$ & $9$ & $12$ & $16$ & $21$ & $28$ & $37$ & $49$ & $65$ & $86$ \\
\hline
$a_{1,3,n}$ & $1$ & $1$ & $2$ & $2$ & $2$ & $3$ & $4$ & $4$ & $5$ & $7$ & $8$ & $9$ & $12$ & $15$ & $17$ & $21$ & $27$  \\
\hline
$a_{1,4,n}$ & $1$ & $1$ & $2$ & $2$ & $2$ & $2$ & $3$ & $4$ & $4$ & $4$ & $5$ & $7$ & $8$ & $8$ & $9$ & $12$ & $15$  \\
\hline
$a_{2,1,n}$ & $0$ & $1$ & $1$ & $1$ & $2$ & $3$ & $4$ & $6$ & $9$ & $13$ & $19$ & $28$ & $41$ & $60$ & $88$ & $129$ & $189$  \\
\hline
$a_{2,3,n}$ & $1$ & $1$ & $1$ & $1$ & $2$ & $2$ & $2$ & $3$ & $3$ & $4$ & $5$ & $5$ & $7$ & $8$ & $9$ & $12$ & $13$  \\
\hline
\end{tabular}
\caption{The first $17$ values of $(a_{p,q, n})_{n=0}^\infty$ with $(p,q) \in \{(1,1), (1,2), (1,3), (1,4), (2,1), (2,3)\}$. Note that for $n\ge 0$, $a_{1,1,n} = \seqnum{A000045}(n+1)$; $a_{1,2,n} = \seqnum{A000931}(n+6)$; $a_{1,3,n}\ =\ \seqnum{A079398}(n+3)$; $a_{1,4,n}\ =\ \seqnum{A103372}(n+4)$; $a_{2,1,n} \ =\ \seqnum{A078012}(n+2)$; $a_{2, 3, n} \ =\ \seqnum{A226503}(n+2)$.}
\label{an}
\end{table}

\begin{table}[H]
\centering
\begin{tabular}{ |c| c| c| c| c| c| c| c| c| c| c| c| c| c| c| c| c| c| c| c| c| c|c|c|}
\hline
$n$ &$0$& $1$ & $2$ & $3$ & $4$ & $5$ & $6$ & $7$ & $8$ & $9$ & $10$ & $11$ & $12$ & $13$ & $14$ & $15$ & $16$ & $17$ \\
\hline
$a^{(m)}_{1,1,n}$ & $1$ & $0$ & $1$ & $1$ & $2$ & $3$ & $5$ & $8$ & $13$ & $21$ & $34$ & $55$ & $89$ & $144$ & $233$ & $377$ & $610$ & $987$ \\
\hline
$a^{(m)}_{1,2,n}$ & $0$ & $0$ & $1$ & $0$ & $1$ & $1$ & $1$ & $2$ & $2$ & $3$ & $4$ & $5$ & $7$ & $9$ & $12$ & $16$ & $21$ & $28$ \\
\hline
$a^{(m)}_{1,3,n}$ & $0$ & $0$ & $1$ & $0$ & $0$ & $1$ & $1$ & $0$ & $1$ & $2$ & $1$ & $1$ & $3$ & $3$ & $2$ & $4$ & $6$ & $5$\\
\hline
$a^{(m)}_{1,4,n}$ & $0$ & $0$ & $1$ & $0$ & $0$ & $0$ & $1$ & $1$ & $0$ & $0$ & $1$ & $2$ & $1$ & $0$ & $1$ & $3$ & $3$ & $1$\\
\hline
$a^{(m)}_{2,1,n}$ & $0$ & $1$ & $0$ & $0$ & $1$ & $1$ & $1$ & $2$ & $3$ & $4$ & $6$ & $9$ & $13$ & $19$ & $28$ & $41$ & $60$ & $88$\\
\hline
$a^{(m)}_{2,3,n}$ & $0$ & $0$ & $0$ & $0$ & $1$ & $0$ & $0$ & $1$ & $0$ & $1$ & $1$ & $0$ & $2$ & $1$ & $1$ & $3$ & $1$ & $3$\\
\hline
\end{tabular}
\caption{The first $18$ values of $(a^{(m)}_{p,q, n})_{n=1}^\infty$ with $(p,q) \in \{(1,1), (1,2), (1,3), (1,4), (2,1), (2,3)\}$. For $n\ge 0$, $a^{(m)}_{1,1,n}\ =\ \seqnum{A212804}(n)$; $a^{(m)}_{1,2, n} = \seqnum{A000931}(n+1)$; $a^{(m)}_{1,3,n+2} = \seqnum{A017817}(n)$; $a^{(m)}_{1,4,n+2} = \seqnum{A017827}(n)$; $a^{(m)}_{2,1,n} = \seqnum{A135851}(n+1)$; $a^{(m)}_{2, 3, n+4} = \seqnum{A052920}(n)$.}
\label{am}
\end{table}

\begin{thm}\label{rc1}
For $(p, q, n)\in \mathbb{N}^3$,  it holds that 
\begin{equation}
|\mathcal{S}^{p/q}_n|\ =\ a_{p,q,(n-1)q}\label{eee18}
\end{equation}
and
\begin{equation}
|\mathcal{M}^{p/q}_{n}|\ =\ a^{(m)}_{p,q, (n-1)q}.\label{eee19}
\end{equation}
\end{thm}

\begin{cor}\label{cc1}
For $(p,q)\in \mathbb{N}^2$, the first $(p+q)$ terms of the sequence $(|\mathcal{S}^{p/q}_n|)_{n=1}^\infty$ are 
$$|\mathcal{S}^{p/q}_n|\ =\ \sum_{i=1}^{\left\lfloor\frac{(n+1)q}{p+q}\right\rfloor}\binom{n-\left\lceil \frac{ni}{q}\right\rceil}{i-1}, \mbox{ for }1\le n\le p+q.$$
Later terms satisfy the recurrence
$$|\mathcal{S}^{p/q}_n| \ =\ \sum_{i=1}^q (-1)^{i+1}\binom{q}{i}|\mathcal{S}^{p/q}_{n-i}| + |\mathcal{S}^{p/q}_{n-(p+q)}|,\mbox{ for }n\ge p+q+1.$$
\end{cor}

\begin{cor}\label{cc2}
For $(p,q)\in \mathbb{N}^2$, the first $(p+q)$ terms of the sequence $(|\mathcal{M}^{p/q}_n|)_{n=1}^\infty$ are 
$$|\mathcal{M}^{p/q}_n|\ =\ \sum_{k=1}^{\left\lfloor\frac{n+1}{p'+q'}\right\rfloor}\binom{n-p'k-1}{q'k-2},$$
where $p' = p/\gcd(p,q)$ and $q' = q/\gcd(p,q)$.
Later terms satisfy the recurrence 
$$|\mathcal{M}^{p/q}_{n}|\ =\ \sum_{i=1}^q (-1)^{i+1}\binom{q}{i}|\mathcal{M}^{p/q}_{n-i}| + |\mathcal{M}^{p/q}_{n-(p+q)}|,\mbox{ for }n\ge p+q+1.$$
\end{cor}

Our paper is structured as follows: Section \ref{formulas} establishes Formulas \eqref{eee13} and \eqref{eee14} for the sequences $(a_{p, q, n})_{n=0}^\infty$ and $(a^{(m)}_{p,q,n})_{n=0}^\infty$; Section \ref{subseq} proves Theorem \ref{rc1}, which states that $(|\mathcal{S}^{p/q}_n|)_{n=1}^\infty$ and $(|\mathcal{M}^{p/q}_n|)_{n=1}^\infty$ are subsequences taken $q$-periodically out of  $(a_{p,q,n})_{n=0}^\infty$ and $(a^{(m)}_{p,q,n})_{n=0}^\infty$, respectively. This enables us to obtain the recurrence of $(|\mathcal{S}^{p/q}_n|)_{n=1}^\infty$ and $(|\mathcal{M}^{p/q}_n|)_{n=1}^\infty$ using characteristic polynomials; Section \ref{relate} investigates the relation between $(|\mathcal{S}^{p/q}_{n}|)_{n=1}^\infty$ and $(|\mathcal{M}^{p/q}_{n}|)_{n=1}^\infty$ when either $p$ or $q$ is equal to $1$ and suggests some problems for future studies. 
%%%%%%%%%%%%%%%%%%%%%%%%%%%%%%%%%%%%%%%%%%%%%%%%%%%%%%%%%%%%%%%%%%%%%%%%%%%%%%%%%%%%%%%%%%%%%%%%%%%%%%%%%%%%%%%%%%%%%%%%%%%%%%%%%%%%%%%%%%%%%%%%%%%%%%%%%%%%%%%%%%%%%%%%%%%%%%%%%%%%%%%%%%%%%%%%%%%%%%%%%%%%%%%%%%%%%%%%%%%%%%%%%%%%%%%%%%%%%%%%%%%%%%%%%%%%%%%%%%%%%%%%%%%%%%%%%%%%%%%%%%%%%%%%%%%%%%%%%%%%%%%%%%%%%%%%%%%%%%%%%%%%%%%%%%%%%%%%%%%%%%%%%%%%%%%%%%%%%%%%%%%%%%%%%%%%
\section{Formulas for $(a_{p, q, n})_{n=0}^\infty$ and $(a^{(m)}_{p,q,n})_{n=0}^\infty$}\label{formulas}
In this section, we prove Formulas \eqref{eee13} and \eqref{eee14} by first confirming that these formulas produce the initial values of $(a_{p, q, n})_{n=0}^\infty$ and $(a^{(m)}_{p,q,n})_{n=0}^\infty$, respectively then showing that later values satisfy the recurrence given by the characteristic polynomial $1-x^q-x^{p+q}$.

\begin{proof}[Proof of \eqref{eee13}]
First, we verify that $\Psi(p, q, n) = a_{p, q,n}$ for $0\le n\le p+q-1$. 

    Case 1: $p\ge q$. When $n\le p-q-1$, the sum 
    $$\sum_{i=0}^{\left\lfloor\frac{n-p+q}{p+q}\right\rfloor}\binom{\left\lfloor \frac{n-p(i+1)}{q}\right\rfloor+1}{i}$$
    is empty and thus, is equal to $0$. When $p-q\le n \le p+q-1$, 
        $$\sum_{i=0}^{\left\lfloor\frac{n-p+q}{p+q}\right\rfloor}\binom{\left\lfloor \frac{n-p(i+1)}{q}\right\rfloor+1}{i}\ =\ \sum_{i=0}^0\binom{\left\lfloor \frac{n-p(i+1)}{q}\right\rfloor+1}{i}\ = \ \binom{\left\lfloor \frac{n-p}{q}\right\rfloor+1}{0} \ =\ 1.$$

    Case 2: $p < q$. When $n\le 2p-1$, 
    $$\sum_{i=0}^{\left\lfloor\frac{n-p+q}{p+q}\right\rfloor}\binom{\left\lfloor \frac{n-p(i+1)}{q}\right\rfloor+1}{i}\ =\ \sum_{i=0}^{0}\binom{\left\lfloor \frac{n-p(i+1)}{q}\right\rfloor+1}{i}\ =\ \binom{\left\lfloor\frac{n-p}{q}\right\rfloor+1}{0}\ =\ 1;$$
    when $2p\le n\le p+q-1$,
    \begin{align*}\sum_{i=0}^{\left\lfloor\frac{n-p+q}{p+q}\right\rfloor}\binom{\left\lfloor \frac{n-p(i+1)}{q}\right\rfloor+1}{i}&\ =\ \sum_{i=0}^1 \binom{\left\lfloor \frac{n-p(i+1)}{q}\right\rfloor+1}{i}\\
    &\ =\ \binom{\left\lfloor\frac{n-p}{q}\right\rfloor+1}{0} + \binom{\left\lfloor\frac{n-2p}
    {q}\right\rfloor+1}{1}\\
    &\ =\ 1 + 1 \ =\ 2.
    \end{align*}

    It remains to prove that $(\Psi(p, q, n))_{n=1}^\infty$ and $(a_{p, q, n})_{n=1}^\infty$  satisfy the same recurrence: for all $n\ge p+q$,
$$\sum_{i=0}^{\left\lfloor\frac{n-p+q}{p+q}\right\rfloor}\binom{\left\lfloor \frac{n-p(i+1)}{q}\right\rfloor+1}{i}\ =\ \sum_{i=0}^{\left\lfloor\frac{n-p}{p+q}\right\rfloor}\binom{\left\lfloor \frac{n-p(i+1)}{q}\right\rfloor}{i} + \sum_{i=0}^{\left\lfloor\frac{n-2p}{p+q}\right\rfloor}\binom{\left\lfloor \frac{n-p(i+2)}{q}\right\rfloor}{i}.$$
We have
\begin{align}\label{re1}
    &\sum_{i=0}^{\left\lfloor\frac{n-p+q}{p+q}\right\rfloor}\binom{\left\lfloor \frac{n-p(i+1)}{q}\right\rfloor+1}{i}- \sum_{i=0}^{\left\lfloor\frac{n-p}{p+q}\right\rfloor}\binom{\left\lfloor \frac{n-p(i+1)}{q}\right\rfloor}{i}\nonumber\\
    \ =\ & \sum_{i=1}^{\left\lfloor\frac{n-p+q}{p+q}\right\rfloor}\binom{\left\lfloor \frac{n-p(i+1)}{q}\right\rfloor+1}{i}- \sum_{i=1}^{\left\lfloor\frac{n-p}{p+q}\right\rfloor}\binom{\left\lfloor \frac{n-p(i+1)}{q}\right\rfloor}{i}\nonumber\\
    \ =\ & \sum_{i=\left\lfloor\frac{n-p}{p+q}\right\rfloor+1}^{\left\lfloor\frac{n-p+q}{p+q}\right\rfloor} \binom{\left\lfloor \frac{n-p(i+1)}{q}\right\rfloor+1}{i}+ \sum_{i=1}^{\left\lfloor\frac{n-p}{p+q}\right\rfloor}\left(\binom{\left\lfloor \frac{n-p(i+1)}{q}\right\rfloor+1}{i}-\binom{\left\lfloor \frac{n-p(i+1)}{q}\right\rfloor}{i}\right)\nonumber\\
    \ =\ & \sum_{i=\left\lfloor\frac{n+q}{p+q}\right\rfloor}^{\left\lfloor\frac{n-p+q}{p+q}\right\rfloor} \binom{\left\lfloor \frac{n-p(i+1)}{q}\right\rfloor+1}{i}+ \sum_{i=1}^{\left\lfloor\frac{n-p}{p+q}\right\rfloor}\binom{\left\lfloor \frac{n-p(i+1)}{q}\right\rfloor}{i-1}\nonumber\\
    \ =\ & \sum_{i=\left\lfloor\frac{n+q}{p+q}\right\rfloor}^{\left\lfloor\frac{n-p+q}{p+q}\right\rfloor} \binom{\left\lfloor \frac{n-p(i+1)}{q}\right\rfloor+1}{i}+ \sum_{i=0}^{\left\lfloor\frac{n-2p-q}{p+q}\right\rfloor}\binom{\left\lfloor \frac{n-p(i+2)}{q}\right\rfloor}{i}.
\end{align}

Write $n - p= (p+q)j + k$ for some $j\ge 0$ and $0\le k\le p+q-1$ and proceed by case analysis.

If $0\le k\le p-1$, then \eqref{re1} is equal to
\begin{align*}\sum_{i=j+1}^{j} \binom{\left\lfloor \frac{n-p(i+1)}{q}\right\rfloor+1}{i}+ \sum_{i=0}^{j-1}\binom{\left\lfloor \frac{n-p(i+2)}{q}\right\rfloor}{i}&\ =\ \sum_{i=0}^{j-1}\binom{\left\lfloor \frac{n-p(i+2)}{q}\right\rfloor}{i}\\
&\ =\ \sum_{i=0}^{\left\lfloor\frac{n-2p}{p+q}\right\rfloor}\binom{\left\lfloor \frac{n-p(i+2)}{q}\right\rfloor}{i}.
\end{align*}

If $p\le k\le p+q-1$, then \eqref{re1} is equal to
\begin{align*}
&\sum_{i=j+1}^{j+1} \binom{\left\lfloor \frac{n-p(i+1)}{q}\right\rfloor+1}{i}+ \sum_{i=0}^{j-1}\binom{\left\lfloor \frac{n-p(i+2)}{q}\right\rfloor}{i}\\
\ =\ &1 + \sum_{i=0}^{j-1}\binom{\left\lfloor \frac{n-p(i+2)}{q}\right\rfloor}{i} \\
\ =\ &\sum_{i=0}^{j}\binom{\left\lfloor \frac{n-p(i+2)}{q}\right\rfloor}{i}\ =\ \sum_{i=0}^{\left\lfloor\frac{n-2p}{p+q}\right\rfloor}\binom{\left\lfloor \frac{n-p(i+2)}{q}\right\rfloor}{i}.
\end{align*}

We have shown that for $n\ge p+q$,
$$\sum_{i=0}^{\left\lfloor\frac{n-p+q}{p+q}\right\rfloor}\binom{\left\lfloor \frac{n-p(i+1)}{q}\right\rfloor+1}{i}- \sum_{i=0}^{\left\lfloor\frac{n-p}{p+q}\right\rfloor}\binom{\left\lfloor \frac{n-p(i+1)}{q}\right\rfloor}{i}\ =\  \sum_{i=0}^{\left\lfloor\frac{n-2p}{p+q}\right\rfloor}\binom{\left\lfloor \frac{n-p(i+2)}{q}\right\rfloor}{i}.$$
\end{proof}

\begin{proof}[Proof of \eqref{eee14}]
First, we verify that $\Phi(p, q, n) = a_{p, q,n}$ for $0\le n\le p+q-1$. 

    Case 1: $p\ge q$. For $n\le p+q-1$, we have 
    $$\frac{n-iq}{p}-2\ \le\ \frac{p+q-1-iq}{p}-2\ \le\  \frac{q-1}{p}-1 \ <\ 0, \mbox{ for all }i\ge 0.$$
    Hence, 
    $$\sum_{\substack{i=0\\ qi\equiv n\mod p}}^{\left\lfloor\frac{n}{q}\right\rfloor}\binom{i}{\frac{n-iq}{p}-2} + \begin{cases}1, &\mbox{ if }n = p-q;\\ 0, &\mbox{ otherwise}\end{cases}\ =\ \begin{cases}1, &\mbox{ if }n = p-q;\\ 0, &\mbox{ if }0\le n \le p+q-1\mbox{ and }n\neq p-q. \end{cases}$$

    Case 2: $p < q$. For $n\le 2p-1$, we have
    $$\frac{n-iq}{p}-2\ \le\ \frac{2p-1}{p}-2\ <\ 0,\mbox{ for all }i\ge 0;$$
    hence
    $$\sum_{\substack{i=0\\ qi\equiv n\mod p}}^{\left\lfloor\frac{n}{q}\right\rfloor}\binom{i}{\frac{n-iq}{p}-2}\ =\ 0.$$
    For $n = 2p$,
    $$\sum_{\substack{i=0\\ qi\equiv n\mod p}}^{\left\lfloor\frac{n}{q}\right\rfloor}\binom{i}{\frac{n-iq}{p}-2}\ =\ \sum_{\substack{i=0\\ p| (qi)}}^{\left\lfloor\frac{2p}{q}\right\rfloor}\binom{i}{\frac{-iq}{p}}\ =\ \binom{0}{0}\ =\ 1$$
    because every positive $i$ makes $\binom{i}{-iq/p} = 0$.
    For $2p+1\le n\le p+q-1$, we have
    $$\frac{n-iq}{p}-2\ \le\ \frac{p+q-1-iq}{p}-2\ =\ \frac{(1-i)q-1}{p}-1\ <\ 0,\mbox{ for all }i\ge 1.$$
    It follows that
    $$\sum_{\substack{i=0\\ qi\equiv n\mod p}}^{\left\lfloor\frac{n}{q}\right\rfloor}\binom{i}{\frac{n-iq}{p}-2}\ =\ \begin{cases} \binom{0}{\frac{n}{p}-2}, &\mbox{ if }p|n\\ 0,&\mbox{ otherwise}\end{cases}\ =\ 0$$
    because if $p|n$, then $n/p\ge 3$.

    Next, we show that for $n\ge p+q$,
    \begin{equation}\label{eee2}\Phi(p, q, n)\ =\ \Phi(p, q, n-q) + \Phi(p, q, n-p-q),\end{equation}
    which is the same recurrence relation as $(a^{(m)}_{p, q, n})_{n=0}^\infty$.
    To do so, we consider the two cases when $n = 2p$ and $n\neq 2p$.

    Case A: $n = 2p$. We have
    $$\Phi(p, q, n-p-q)\ =\ \Phi(p, q, p-q)\ = \ \sum_{\substack{i=0\\ p|((i+1)q)}}^{\left\lfloor\frac{p}{q}\right\rfloor-1}\binom{i}{\frac{-(i+1)q}{p}-1} + 1\ =\ 1$$
    because $-(i+1)q/p-1 < 0$ for all $i\ge 0$.
    Then $$\Phi(p, q, 2p)\ =\ \Phi(p, q, 2p-q) + \Phi(p, q, p-q)$$ can be rewritten as
    \begin{equation}\label{eee3}\sum_{\substack{i=0\\ p|(qi)}}^{\left\lfloor\frac{2p}{q}\right\rfloor}\binom{i}{\frac{-iq}{p}}\ =\ \sum_{\substack{i=0\\ p|((i+1)q)}}^{\left\lfloor\frac{2p}{q}\right\rfloor-1}\binom{i}{\frac{-(i+1)q}{p}}+1.\end{equation}
    On the left side of \eqref{eee3}, only $i = 0$ contributes a nonzero term to the sum, so $$\sum_{\substack{i=0\\ p|(qi)}}^{\left\lfloor\frac{2p}{q}\right\rfloor}\binom{i}{\frac{-iq}{p}} \ =\ 1.$$
    On the right side of \eqref{eee3}, each binomial in the sum is zero because $i \ge 0 >  -(i+1)q/p$. This confirms \eqref{eee2} when $n = 2p$.

    Case B: $n\neq 2p$. We write \eqref{eee2} as 
    \begin{align}\label{eee4}&\sum_{\substack{i=0\\ qi\equiv n\mod p}}^{\left\lfloor\frac{n}{q}\right\rfloor}\binom{i}{\frac{n-iq}{p}-2}\nonumber\\
    \ =\ &\sum_{\substack{i=0\\ q(i+1)\equiv n\mod p}}^{\left\lfloor\frac{n}{q}\right\rfloor-1}\binom{i}{\frac{n-(i+1)q}{p}-2} + \sum_{\substack{i=0\\ q(i+1)\equiv n\mod p}}^{\left\lfloor\frac{n-p}{q}\right\rfloor-1}\binom{i}{\frac{n-(i+1)q}{p}-3}.\end{align}
    The right side of \eqref{eee4} is 
    \begin{align}\label{eee5}
        &\sum_{\substack{i=\left\lfloor\frac{n-p}{q}\right\rfloor\\ q(i+1)\equiv n\mod p}}^{\left\lfloor\frac{n}{q}\right\rfloor-1}\binom{i}{\frac{n-(i+1)q}{p}-2} + \sum_{\substack{i=0\\ q(i+1)\equiv n\mod p}}^{\left\lfloor\frac{n-p}{q}\right\rfloor-1}\left(\binom{i}{\frac{n-(i+1)q}{p}-2}+\binom{i}{\frac{n-(i+1)q}{p}-3}\right)\nonumber\\
        \ =\ &\sum_{\substack{i=\left\lfloor\frac{n-p}{q}\right\rfloor \\ q(i+1)\equiv n\mod p}}^{\left\lfloor\frac{n}{q}\right\rfloor-1}\binom{i}{\frac{n-(i+1)q}{p}-2} + \sum_{\substack{i=0\\ q(i+1)\equiv n\mod p}}^{\left\lfloor\frac{n-p}{q}\right\rfloor-1}\binom{i+1}{\frac{n-(i+1)q}{p}-2}\nonumber\\
        \ =\ &\sum_{\substack{i=\left\lfloor\frac{n-p}{q}\right\rfloor \\ q(i+1)\equiv n\mod p}}^{\left\lfloor\frac{n}{q}\right\rfloor-1}\binom{i}{\frac{n-(i+1)q}{p}-2} + \sum_{\substack{i=1\\ qi\equiv n\mod p}}^{\left\lfloor\frac{n-p}{q}\right\rfloor}\binom{i}{\frac{n-iq}{p}-2}.
    \end{align}
    It follows from \eqref{eee4} and \eqref{eee5} that we need to prove
    \begin{equation}\label{eee6}
        \sum_{\substack{i=\left\lfloor\frac{n-p}{q}\right\rfloor+1\\ qi\equiv n\mod p}}^{\left\lfloor\frac{n}{q}\right\rfloor}\binom{i}{\frac{n-iq}{p}-2}\ =\       \sum_{\substack{i=\left\lfloor\frac{n-p}{q}\right\rfloor\\ q(i+1)\equiv n\mod p}}^{\left\lfloor\frac{n}{q}\right\rfloor-1}\binom{i}{\frac{n-(i+1)q}{p}-2}.
    \end{equation}
    Note that we discard $\binom{0}{n/p-2}$ on the left side of \eqref{eee4} because the term is nonzero only when $n = 2p$. We claim that both sides of \eqref{eee6} vanish. On the left, we have
    $$\frac{n-p}{q}\ <\ i\ \le\ \frac{n}{q}\quad\mbox{ and }\quad qi\equiv n\mod p;$$
    hence,
    $$-p\ <\ qi-n \ \le\ 0\quad\mbox{ and }\quad qi\equiv n\mod p.$$
    It follows that $qi -n = 0$, and the corresponding binomial is $\binom{i}{-2} = 0$ because $i\ge 1$. Similarly, the right side of \eqref{eee6} is also $0$.
\end{proof}

\section{Subsequences of Padovan-like sequences}\label{subseq}

\subsection{Proof of Theorem \ref{rc1}}
In this section, we show that the terms of $(|\mathcal{S}^{p/q}_{n}|)_{n=1}^\infty$ and $(|\mathcal{M}^{p/q}_{n}|)_{n=1}^\infty$ are taken $q$-periodically from the sequences $(a_{p,q,n})_{n=0}^\infty$ and $(a^{(m)}_{p,q,n})_{n=0}^\infty$. This allows us to deduce a linear recurrence for $(|\mathcal{S}^{p/q}_{n}|)_{n=1}^\infty$ and $(|\mathcal{M}^{p/q}_{n}|)_{n=1}^\infty$.

First, we find a formula for $|\mathcal{S}^{p/q}_n|$. For $1\le i\le n$, let $F$ be an $i$-element set in $\mathcal{S}^{p/q}_n$, i.e., $F\subset \{1, 2, \ldots, n\}$, $n\in F$, $|F| = i$, and $q\min F\ge pi$. Hence, to form $F$, we choose $i-1$ elements in 
$\left\{\left\lceil \frac{pi}{q}\right\rceil, \ldots, n-1\right\}$. It follows that the number of $i$-element sets in $|\mathcal{S}^{p/q}_n|$ is $$\binom{n-\left\lceil \frac{pi}{q}\right\rceil}{i-1}.$$
Here we require that $n - \left\lceil pi/q\right\rceil\ge i-1$, or equivalently, $i\le \left\lfloor\frac{(n+1)q}{p+q}\right\rfloor$. Therefore, we have
\begin{equation}\label{eee20'}|\mathcal{S}^{p/q}_n|\ =\ \sum_{i=1}^{\left\lfloor\frac{(n+1)q}{p+q}\right\rfloor}\binom{n-\left\lceil \frac{pi}{q}\right\rceil}{i-1}.\end{equation}

Next, we give a formula for $|\mathcal{M}^{p/q}_n|$. 
Let $p' = p/\gcd(p,q)$ and $q' = q/\gcd(p,q)$. Consider a $(p,q)$-maximal Schreier set $F\subset \{1, \ldots, n\}$ with $n\in F$ and $|F| = i$. Since $q\min F = p|F| = pi$, $i$ is a multiple of $q'$. Let $i = q'k$. Then $\min F = p'k$, and $F = \{p'k, n\}\cup F'$, where $F'\subset \{p'k + 1, p'k + 2, \ldots, n-1\}$ and $|F'| = q'k-2$. Hence, there are $\binom{n-p'k-1}{q'k-2}$ such sets $F$. Here we require $n - p'k-1\ge q'k-2$, i.e., $k\le (n+1)/(p'+q')$. Therefore,
\begin{equation}\label{eee20}|\mathcal{M}^{p/q}_{n}|\ =\ \sum_{k=1}^{\left\lfloor\frac{n+1}{p'+q'}\right\rfloor}\binom{n-p'k-1}{q'k-2}.\end{equation}

\begin{proof}[Proof of Theorem \ref{rc1}]
First, we prove \eqref{eee18}. We have 
\begin{align*}
\Psi(p, q, (n-1)q)&\ =\ \sum_{i=0}^{\left\lfloor\frac{(n-1)q-p+q}{p+q}\right\rfloor}\binom{\left\lfloor \frac{(n-1)q-p(i+1)}{q}\right\rfloor+1}{i}\\
&\ =\ \sum_{i=0}^{\left\lfloor\frac{nq-p}{p+q}\right\rfloor}\binom{n + \left\lfloor\frac{-p(i+1)}{q}\right\rfloor}{i}\\
&\ =\ \sum_{i=0}^{\left\lfloor\frac{nq-p}{p+q}\right\rfloor}\binom{n - \left\lceil\frac{p(i+1)}{q}\right\rceil}{i}\\
&\ =\ \sum_{i=1}^{\left\lfloor\frac{(n+1)q}{p+q}\right\rfloor}\binom{n - \left\lceil\frac{pi}{q}\right\rceil}{i-1}\ =\ |\mathcal{S}^{p/q}_n|.
\end{align*}

Next, we prove \eqref{eee19}. Let $p' = p/\gcd(p,q)$ and $q' = q/\gcd(p,q)$.
We have
\begin{align}\label{eee10}&\Phi(p, q, (n-1)q)\nonumber\\
\ =\ &\sum_{\substack{i=0\\ p'|(n-i-1)}}^{n-1}\binom{i}{\frac{q'(n-i-1)}{p'}-2} + \begin{cases}1, &\mbox{ if } (n,q') = (p',1),\\ 0, &\mbox{ otherwise,}\end{cases}\nonumber\\
    \ =\ &\sum_{k=0}^{\left\lfloor\frac{n-1}{p'}\right\rfloor}\binom{n-p'k-1}{q'k-2} +\begin{cases}1, &\mbox{ if } (n,q') = (p',1),\\ 0, &\mbox{ otherwise,}\end{cases} \quad \mbox{(with } n-i-1 = p'k\mbox{)}  \nonumber\\
    \ =\ &\sum_{k=1}^{\left\lfloor\frac{n-1}{p'}\right\rfloor}\binom{n-p'k-1}{q'k-2} + \begin{cases}1, &\mbox{ if } (n,q') = (p',1),\\ 0, &\mbox{ otherwise.}\end{cases}
\end{align}
Thanks to \eqref{eee20} and \eqref{eee10}, it suffices to verify that for $n\in \mathbb{N}$,
\begin{equation}\label{eee12}\sum_{k=1}^{\left\lfloor\frac{n+1}{p'+q'}\right\rfloor}\binom{n-p'k-1}{q'k-2}\ =\ \sum_{k=1}^{\left\lfloor\frac{n-1}{p'}\right\rfloor}\binom{n-p'k-1}{q'k-2} + \begin{cases}1, &\mbox{ if } (n,q') = (p',1),\\ 0, &\mbox{ otherwise.}\end{cases}\end{equation}
We proceed by case analysis.

Case 1: $(n-1)/p'\ge (n+1)/(p'+q')$, i.e., $nq' \ge 2p' + q'$. Then $(n,q')\neq (p',1)$. We have
$$\sum_{k=1}^{\left\lfloor\frac{n-1}{p'}\right\rfloor}\binom{n-p'k-1}{q'k-2} - \sum_{k=1}^{\left\lfloor\frac{n+1}{p'+q'}\right\rfloor}\binom{n-p'k-1}{q'k-2}\ =\ \sum_{k=\left\lfloor\frac{n+1}{p'+q'}\right\rfloor+1}^{\left\lfloor\frac{n-1}{p'}\right\rfloor}\binom{n-p'k-1}{q'k-2}.$$
For each $k\ge \left\lfloor \frac{n+1}{p'+q'}\right\rfloor+1$, $n-p'k-1 < q'k - 2$, so the binomial $\binom{n-p'k-1}{q'k-2}$ is zero unless $q'k = 2$, which occurs when either $(q',k) = (2,1)$ or $(q',k) = (1,2)$.

\begin{enumerate}
\item[a)] Case 1.1: $(q',k) = (2,1)$. It follows from $k = 1$ that $\left\lfloor \frac{n+1}{p'+q'}\right\rfloor = 0$, which implies $n < p' + 1$. However, $nq' \ge 2p' + q'$ gives $2n \ge 2p' + 2$, which contradicts $n < p'+1$. 

\item[b)] Case 1.2: $(q',k) = (1,2)$. Then  $\left\lfloor \frac{n+1}{p'+q'}\right\rfloor = 1$, which implies 
$$n\ <\ 2p'+2q'-1\ =\ 2p'+1.$$
However, $nq'\ge 2p'+q'$ gives $n\ge 2p'+1$, which contradicts $n < 2p'+1$. 
\end{enumerate}

Case 2: $(n-1)/p' < (n+1)/(p'+q')$, i.e., $n < 1 + 2p'/q'$. Then $(n+1)/(p'+q') < 2$.

If $\left\lfloor\frac{n+1}{p'+q'}\right\rfloor = 0$, then $(n,q')\neq (p',1)$, and $(n-1)/p' < (n+1)/(p'+q')$ implies that $\left\lfloor\frac{n-1}{p'}\right\rfloor = 0$. Hence, both sides of \eqref{eee12} are zero. 

If $\left\lfloor\frac{n+1}{p'+q'}\right\rfloor = 1$, then $\left\lfloor\frac{n-1}{p'}\right\rfloor \in \{0, 1\}$. 
\begin{enumerate}
    \item[a)] Case 2.1: If 
$\left\lfloor\frac{n-1}{p'}\right\rfloor = 1$, then $n\neq p'$, and \eqref{eee12} holds.
\item[b)] Case 2.2: If $\left\lfloor\frac{n-1}{p'}\right\rfloor = 0$, then $n< p'+1$. On the other hand,  $\left\lfloor\frac{n+1}{p'+q'}\right\rfloor = 1$ implies that $n\ge p'+q'-1$. Hence,
$$p'+1 \ >\ n\ \ge\ p'+q'-1\ \Longrightarrow\ q' = 1.$$
Then \eqref{eee12} becomes
$$\sum_{k=1}^{1}\binom{n-p'k-1}{k-2}\ =\ \sum_{k=1}^{0}\binom{n-p'k-1}{k-2}+ \begin{cases}1, &\mbox{ if }n=p',\\ 0, &\mbox{ otherwise.}\end{cases}$$
Equivalently,
$$\binom{n-p-1}{-1}\ =\ \begin{cases}1, &\mbox{ if }n=p,\\ 0, &\mbox{ otherwise,}\end{cases}$$
which is true.
\end{enumerate}
\end{proof}

\subsection{Recurrence of $(|\mathcal{S}^{p/q}_{n}|)_{n=1}^\infty$ and $(|\mathcal{M}^{p/q}_n|)_{n=1}^\infty$}
By Theorem \ref{rc1}, the recurrence of $(|\mathcal{S}^{p/q}_n|)_{n=1}^\infty$ and $(|\mathcal{M}^{p/q}_{n}|)_{n=1}^\infty$ is the same as the recurrence of subsequences taken $q$-periodically out of $(a_{p,q,n})_{n=0}^\infty$ and $(a^{(m)}_{p,q, n})_{n=0}^\infty$, respectively.

\begin{defi}\normalfont
Let $p(x) = c_0+c_1x+\cdots+c_rx^r$ be a polynomial with real coefficients $(c_i)_{i=0}^r$. A sequence $(a_n)_{n=0}^\infty$ is said to \textit{satisfy} $p(x)$ if 
$$c_0a_n + c_1a_{n-1} + \cdots + c_ra_{n-r}\ =\ 0,\mbox{ for all }n\ge r.$$
Trivially, all sequences satisfy the zero polynomial. 
\end{defi}

We need the following lemma, whose proof is straightforward and therefore omitted.

\begin{lem}\label{dl}
If a sequence $(a_n)_{n=0}^\infty$ satisfies $p(x)$ and $p(x)$ divides $q(x)$, then $(a_n)_{n=0}^\infty$ satisfies $q(x)$. 
\end{lem}

\begin{proof}[Proof of Corollaries \ref{cc1} and \ref{cc2}] We verify initial terms and prove the linear recurrence for $(|\mathcal{S}^{p/q}_n|)_{n=1}^\infty$ and $(|\mathcal{M}^{p/q}_n|)_{n=1}^\infty$. 
The formulas for initial terms are due to \eqref{eee20'} and \eqref{eee20}. We prove the recurrence relation.
By definition, both $(a_{p,q,n})_{n=0}^\infty$ and $(a^{(m)}_{p,q,n})_{n=0}^\infty$ satisfy $g_{p,q}(x):=1-x^q-x^{p+q}$. Define 
$$h_{p,q}(x)\ :=\ 1 - \sum_{i=1}^q (-1)^{i+1}\binom{q}{i}x^i - x^{p+q}.$$
We have
\begin{align*}
h_{p,q}(x^q) &\ =\ 1 + \sum_{i=1}^q (-1)^{i}\binom{q}{i}x^{qi} - x^{q(p+q)}\\
&\ =\ \sum_{i=0}^q (-1)^i \binom{q}{i}x^{qi} - x^{q(p+q)}\\
&\ =\ (1-x^q)^q - (x^{p+q})^q\\
&\ =\ (1-x^q-x^{p+q})\sum_{i=0}^{q-1} (1-x^q)^{q-1-i}x^{(p+q)i},
\end{align*}
so $g_{p,q}(x)$ divides $h_{p,q}(x^q)$. By Lemma \ref{dl}, $(a_{p,q,n})_{n=0}^\infty$ and $(a^{(m)}_{p,q,n})_{n=0}^\infty$ satisfies $h_{p,q}(x^q)$. It follows from Theorem \ref{rc1} that $(|\mathcal{S}^{p/q}_n|)_{n=1}^\infty$ and $(|\mathcal{M}^{p/q}_n|)_{n=1}^\infty$ satisfy $h_{p,q}(x)$. 
Here we use the obvious fact that for a fixed $k\in \mathbb{N}$, if $(b_n)_{n=1}^\infty$ satisfies 
$c_0 + c_k x^k + c_{2k}x^{2k} + \cdots + c_{nk} x^{nk}$,
then the subsequence $(b_{(n-1)k})_{n=1}^\infty$ satisfies 
$c_0 + c_k x + c_{2k}x^2 + \cdots + c_{nk}x^n$.\end{proof}

\begin{rek}\normalfont
While one goal of the present paper is to establish the recurrence relation of $(|\mathcal{S}_n^{p/q}|)_{n=1}^\infty$ and $(|\mathcal{M}^{p/q}_n|)_{n=1}^\infty$ without using the Inclusion-Exclusion Principle as in \cite{BCF}, we can use the same argument in the proof of Corollaries \ref{cc1} and \ref{cc2} to show that \eqref{eee18} is a corollary of \cite[Theorem 1.1]{BCF}. Indeed, \eqref{ee1} states that
$(|\mathcal{S}^{p/q}_n|)_{n=1}^\infty$ satisfies 
$$h_{p,q}(x)\ =\ 1 + \sum_{i=1}^q (-1)^{i}\binom{q}{i}x^i - x^{p+q}\ =\ (1-x)^q - x^{p+q}.$$
By definition, $(a_{p,q,n})_{n=0}^\infty$ satisfies $g_{p,q}(x) = 1 - x^q - x^{p+q}$. Since $g_{p,q}(x)$ divides $h_{p,q}(x^q) = (1-x^q)^q - x^{q(p+q)}$, we know that $(a_{p,q,n})_{n=0}^\infty$ satisfies
$h_{p,q}(x^q)$. It follows that $(a_{p,q,(n-1)q})_{n=1}^\infty$ satisfies $h_{p,q}(x)$. Since $(|\mathcal{S}^{p/q}_n|)_{n=1}^\infty$ and $(a_{p,q,(n-1)q})_{n=1}^\infty$ satisfy the same polynomial and have the same initial terms, the two sequences are identical. 
\end{rek}

\section{Relation between $(|\mathcal{S}_n^{p/q}|)_{n=1}^\infty$ and $(|\mathcal{M}^{p/q}_n|)_{n=1}^\infty$ and future investigations}\label{relate}

It is natural to ask about the relation between the two sequences $(|\mathcal{S}_n^{p/q}|)_{n=1}^\infty$ and $(|\mathcal{M}^{p/q}_n|)_{n=1}^\infty$: given one of the sequences, can we compute the other? This section presents the two special cases when $p = 1$ or $q = 1$ of the following problem.

\begin{prob}\normalfont
    For $(p,q) \in \mathbb{N}^2$, find the constants $(c_i)_{i=1}^\infty$ (if any) such that
    $$|\mathcal{M}^{p/q}_n| = \sum_{i=1}^\infty c_i |\mathcal{S}^{p/q}_i|, \mbox{ for all }n\in \mathbb{N}.$$
\end{prob}

\begin{thm}
    For $p, n\in \mathbb{N}$, it holds that 
    $$|\mathcal{S}_n^{p/1}|\ =\ |\mathcal{M}^{p/1}_{n+p+1}|.$$
\end{thm}

\begin{proof}
    By \eqref{eee20'} and \eqref{eee20}, we have
    $$|\mathcal{S}_n^{p/1}|\ =\ \sum_{i=1}^{\left\lfloor\frac{n+1}{p+1}\right\rfloor}\binom{n-pi}{i-1}\mbox{ and }|\mathcal{M}_n^{p/1}|\ =\ \sum_{i=1}^{\left\lfloor\frac{n+1}{p+1}\right\rfloor}\binom{n-pi-1}{i-2}.$$
    Hence,
    \begin{align*}|\mathcal{M}_{n+p+1}^{p/1}|\ =\ \sum_{i=1}^{\left\lfloor\frac{n+p+2}{p+1}\right\rfloor}\binom{n+p+1-pi-1}{i-2}&\ =\ \sum_{i=1}^{\left\lfloor\frac{n+1}{p+1}\right\rfloor+1}\binom{n-p(i-1)}{i-2}\\
    &\ =\ \sum_{i=0}^{\left\lfloor\frac{n+1}{p+1}\right\rfloor}\binom{n-pi}{i-1}\\
    &\ =\ \sum_{i=1}^{\left\lfloor\frac{n+1}{p+1}\right\rfloor}\binom{n-pi}{i-1}\ =\ |\mathcal{S}_n^{p/1}|.
    \end{align*}
\end{proof}

\begin{thm}\label{m3}
For $q, n\in \mathbb{N}$, it holds that
\begin{equation}\label{e10q0}|\mathcal{M}_n^{1/q}|\ =\ 2|\mathcal{S}_n^{1/q}| - |\mathcal{S}^{1/q}_{n+1}|.\end{equation}
\end{thm}

\begin{proof}
Let $n\in \mathbb{N}$. When $q = 1$, $|\mathcal{S}_n^{1/1}| = F_n$, $|\mathcal{S}_{n+1}^{1/1}| = F_{n+1}$, and $|\mathcal{M}_n^{1/1}| = F_{n-2}$ (see \eqref{e0} and \eqref{e1p1}). Then \eqref{e10q0} holds. For the rest of the proof, assume that $q\ge 2$. 

We need to bijectively map two copies of $\mathcal{S}^{1/q}_{n}$ to $\mathcal{S}^{1/q}_{n+1}\cup \mathcal{M}^{1/q}_{n}$. 
Observe that 
\begin{align*}
\phi: \mbox{ }\mathcal{S}^{1/q}_{n}\ &\rightarrow\ \mathcal{S}^{1/q}_{n+1}\\ 
F\ &\mapsto\ F+q
\end{align*}
maps $\mathcal{S}^{1/q}_{n}$ injectively into $\mathcal{S}^{1/q}_{n+1}$. Furthermore, $\mathcal{M}^{1/q}_{n}\subset \mathcal{S}^{1/q}_{n}$. Hence, it suffices to show that 
\begin{equation}\label{e1}|\mathcal{S}^{1/q}_{n}\backslash \mathcal{M}^{1/q}_{n}|\ =\ |\mathcal{S}^{1/q}_{ n+1}\backslash \phi(\mathcal{S}^{1/q}_{n})|.\end{equation}

The left side of \eqref{e1} is the cardinality of 
$$L_{q,n}\ :=\ \{F\subset \{1, 2, \ldots, n\}\,:\, q\min F > |F|\mbox{ and }\max F = n\},$$
which can be partitioned into the collections
$$L_{q, n,i}\ :=\ \{F\subset \{1, 2, \ldots, n\}\,:\, q\min F > |F| = i\mbox{ and }\max F = n\}, \mbox{ with }i\ge 1.$$
For $i = 1$, the only set in $L_{q, n,1}$ is $\{n\}$, so $|L_{q, n, 1}| = 1$. For $i\ge 2$, to form sets in $L_{q, n, i}$, we first choose the minimum $j$ with $i/q < j\le n-1$ then choose $i-2$ integers in $[j+1, n-1]$, which requires  $i-2\le n-j-1$. Hence, for $i\ge 2$, 
$$|L_{q, n, i}|\ =\ \sum_{\substack{i+1\le jq\le (n-1)q\\ j\le n-i+1}}\binom{n-j-1}{i-2}\ =\ \sum_{i+1\le jq\le (n-i+1)q}\binom{n-j-1}{i-2}.$$

The right side of \eqref{e1} is the cardinality of
$$R_{q, n}\ :=\ \{F\subset \{1, 2, \ldots, n+1\}\,:\, q\min F \ge |F| > q\min F - q, \max F = n+1\},$$
which can be partitioned into 
\begin{align*}
    R_{q, n, i}\ :=\  \{F\subset \{1, 2, \ldots, n+1\}\,:\, q\min F &\ge |F| = i > q\min F - q,\\
    &\max F = n+1\},\mbox{ with }i\ge 1.
\end{align*}
We have $R_{q, n, 1} = \emptyset$ because a nonempty $R_{q, n,1}$ requires $1 > (n+1)q-q = nq\ge 2$. For $i = 2$, an
$F\in R_{q, n, 2}$ must have the form $F = \{j, n+1\}$ with $1\le j\le n$ and $jq\ge 2 > (j-1)q$. Since $q\ge 2$, the inequality $2 > (j-1)q$ implies that $j=1$. Hence, $R_{q, n, 2} = \{1, n+1\}$ and $|R_{q, n, 2}| = 1$. To form sets in $R_{q, n, i}$ with $i \ge 3$, we first choose the minimum $j$ with $j-1 < i/q \le j\le n$ then choose $i-2$ integers in $[j+1, n]$, which requires
$n-j\ge i-2$. Hence, for $i\ge 3$,
$$|R_{q, n, i}| \ =\ \sum_{\substack{1+q(j-1) \le i\le jq\le nq\\j\le n+2-i}}\binom{n-j}{i-2}\ =\ \sum_{\substack{i\le jq\le i+q-1\\ j\le n-i+2}}\binom{n-j}{i-2}.
$$

To prove \eqref{e1}, we show that 
\begin{equation}\label{e2}|L_{q, n, i}|\ =\ |R_{q, n, i+1}|, \mbox{ for }i\ge 1.\end{equation}
The above analysis shows that \eqref{e2} holds when $i = 1$. For $i\ge 2$, \eqref{e2} is the same as
\begin{equation}\label{e3}\sum_{i+1\le jq\le (n-i+1)q}\binom{n-j-1}{i-2}\ =\ \sum_{\substack{i+1\le jq\le i+q\\ j\le n-i+1}}\binom{n-j}{i-1}.\end{equation}

Case 1: $(n-i+1)q\le i+q$. We have
\begin{align}\label{e21}
    &\sum_{i+1\le jq\le (n-i+1)q}\binom{n-j-1}{i-2}-\sum_{\substack{i+1\le jq\le i+q\\ j\le n-i+1}}\binom{n-j}{i-1}\nonumber\\
    &\ =\  \sum_{i+1\le jq\le (n-i+1)q}\binom{n-j-1}{i-2}-\sum_{i+1\le jq\le (n-i+1)q}\binom{n-j}{i-1}.
\end{align}
Since $(n-i+1)q\le i+q$, 
$$n\ \le\ i+\frac{i}{q}\ =\ i+\frac{(i+1)-1}{q}\ \le\ i+\frac{jq-1}{q}\ <\ i+j.$$
Hence, $n\le i + j-1$, but \eqref{e21} sums over $j\le n-i+1$. Therefore, $n = i + j - 1$ in \eqref{e21}, which gives
\begin{align*}&\sum_{i+1\le jq\le (n-i+1)q}\binom{n-j-1}{i-2}-\sum_{i+1\le jq\le (n-i+1)q}\binom{n-j}{i-1}\\
&\ =\ \sum_{i+1\le jq\le (n-i+1)q} 1 - \sum_{i+1\le jq\le (n-i+1)q}1\ =\ 0.
\end{align*}

Case 2: $(n-i+1)q > i+q$. Then \eqref{e3} is equivalent to
\begin{equation*}\sum_{i+1\le jq\le (n-i+1)q}\binom{n-j-1}{i-2}\ =\ \binom{n-\lceil\frac{i+1}{q}\rceil}{i-1},\end{equation*}
which is precisely the hockey-stick identity (see \cite[Theorem 1.2.3 item (5)]{We}).

We have verified \eqref{e3} and thus completed the proof.
\end{proof}

\section{Acknowledgment}
The authors would like to thank the anonymous referee for a careful reading of the paper and for the helpful suggestions.

This work is partially supported by the College of Arts \& Sciences at Texas A\&M University. The second named author is an undergraduate at Texas A\&M University, working under the guidance of the first named author.

%%%%%%%%%%%%%%%%%%%%%%%%%%%%%%%%%%%%%%%%%%%%%%%%%%%%%%%%%%%%%%%%%%%%%%%%%%%%%%%%%%%%%%%%%%%%%%%%%%%%%%%%%%%%%%%%%%%%%%%%%%%%%%%%%%%%%%%%%%%%%%%%%%%%%%%%%%%%%%%%%%%%%%%%%%%%%%%%%%%%%%%%%%%%%%%%%%%%%%%%%%%%%%%%%%%%%%%%%%%%%%%%%%%%%%%%%%%%%%%%%%%%%%%%%%%%%%%%%%%%%%%%%%%%%%%%%%%%%%%%%%%%%%%%%%%%%%%%%%%%%%%%%%%%%%%%%%%%%%%%%%%%%%%%%%%%%%%%%%%%%%%%%%%%%%%%%%%%%%%%%%%%%%%%%%%%

\ \\
\end{document}